\newif\ifpictures
\picturesfalse

\documentclass[12pt]{amsart}
\usepackage{amssymb,amsmath}
 \usepackage{amsopn}
 \usepackage{xspace}
  \usepackage{hyperref}
 \usepackage[dvips]{graphicx}
\usepackage[arrow,matrix,curve]{xy}
\usepackage {color, tikz}

\numberwithin{equation}{section}
\newtheorem{thm}{Theorem}

\newtheorem{lemma}[thm]{Lemma}
\newtheorem{cor}[thm]{Corollary}
\newtheorem{conj}[thm]{Conjecture}
\newtheorem{defn}[thm]{Definition}

\numberwithin{thm}{section}

\newcounter{FNC}[page]
\def\newfootnote#1{{\addtocounter{FNC}{2}$^\fnsymbol{FNC}$%
     \let\thefootnote\relax\footnotetext{$^\fnsymbol{FNC}$#1}}}

\newcommand{\N}{\mathbb{N}}

\definecolor{DarkGreen}{rgb}{0,0.65,0}
\definecolor{MyPurple}{rgb}{0.9,0,0.9}
\definecolor{MyYellow}{rgb}{0.6,0.6,0}

\DeclareMathOperator{\cp}{cp}

\title[A Sharp Upper Bound for the Complexity of Labeled Oriented Trees]{A Sharp Upper Bound for the Complexity of Labeled Oriented Trees}
\author{Moritz Christmann} \author{Timo de Wolff}

\address{Moritz Christmann, Goethe University, D-60054 Frankfurt am Main, Germany\medskip \newline 
\hspace*{8pt} Timo de Wolff, Texas A\&M University, Department of Mathematics, College Station, TX 77843-3368, USA \medskip}

\email{christma@stud.uni-frankfurt.de \\
dewolff@math.tamu.edu}

\subjclass[2010]{05C10, 05C25, 20F65, 57M15, 57M20}
\keywords{Aspherical, Labeled Oriented Tree, LOT, Whitehead Conjecture}

\begin{document}

\begin{abstract}
A labeled oriented graph (LOG) is an oriented graph with a labeling function from the edge set into the vertex set. The complexity of a LOG is the minimal cardinality of an initial set $S$ of vertices such that every vertex can be reached successively from $S$ only using edges with labels in $S$ or already visited vertices. We give a constructive proof of a conjecture by Rosebrock stating that for an interior reduced, connected LOG with $m$ vertices the complexity is at most $(m+1)/2$ and show that this bound is sharp.

Due to results of Howie labeled oriented trees (LOTs) yield crucial candidates for counterexamples of the Whitehead Conjecture stating that every subcomplex of an aspherical 2-complex is aspherical. We explicitly describe the structure of LOTs of maximal complexity $(m+1)/2$. We conclude that the 2-complexes associated to these LOTs are always aspherical excluding them from the list of possible counterexamples.
\end{abstract}

\maketitle

\section{Introduction}

Let $G$ be a finitely presented group with presentation $P := \langle x_1,\ldots,x_n \ | \ R_1,\ldots,R_s \rangle$. Assume, every relation $R_j$ is of the form $x_k x_i x_k^{-1} = x_j$ with $x_i \neq x_j$. We associate a graph $\Gamma(P)$ to $P$ in the following way:
\begin{itemize}
 \item We define the vertex set as $V(\Gamma(P)) = \{x_1,\ldots,x_n\}$, i.e., we introduce one vertex for every generator.
 \item For every relation $x_k x_i x_k^{-1} = x_j$ we add a directed edge from $x_i$ to $x_j$.
 \item We define a labeling function $l: E(\Gamma(P)) \to V(\Gamma(P))$. Namely, if $e$ is the edge from $x_i$ to $x_j$ induced by the relation $x_k x_i x_k^{-1} = x_j$, then we set $l(e) := x_k$.
\end{itemize}
We say that $\Gamma(P)$ is a \textit{labeled oriented graph (LOG)}. And if $\Gamma(P)$ is a tree, then we call $\Gamma(P)$ a \textit{labeled oriented tree (LOT)}. If the context is clear, then we simply write $\Gamma$. See Figure \ref{Fig:LOTExample} for an example.

Note that analogously every LOG $\Gamma$ yields a \textit{LOG presentation} $P(\Gamma)$, i.e., a group presentation in the upper sense. Therefore, we call each $x_j$ simultaneously a \textit{vertex} and a \textit{generator} each $e_{ij}$ simultaneously an \textit{edge} and a \textit{relation} respectively. Following the literature \cite{Harlander:Rosebrock} we call a LOG $\Gamma$ \textit{interior reduced} if the label $l(e_{ij})$ of every edge $e_{ij} = (x_i,x_j)$ satisfies $l(e_{ij}) \notin \{x_i,x_j\}$. See \cite{Howie:Whitehead,Howie:RibbonDiscs, Rosebrock:WhiteheadOverview, Rosebrock:LOTComplexity} for more information.

\begin{figure}
\begin{tikzpicture}
   \draw [->] (2,0) -- (2,1);
   \draw (2,1) -- (2,2);
   \draw [->] (0,2) -- (1,1);
   \draw (1,1) -- (2,0);
   \draw [->] (2,0) -- (3,1);
   \draw (3,1) -- (4,2);
   
   \draw [fill] (2,0) circle [radius=0.1];
   \draw [fill] (0,2) circle [radius=0.1];
   \draw [fill] (4,2) circle [radius=0.1];
   \draw [fill] (2,2) circle [radius=0.1];
   
   \node [above] at (0,2.2) {$a$};
   \node [above] at (2,2.2) {$b$};
   \node [above] at (4,2.2) {$c$};
   \node [below] at (2,-0.3) {$d$};
   
   \node [right] at (1.1,1) {$c$};
   \node [right] at (2,1) {$c$};
   \node [right] at (3.1,1) {$a$};
\end{tikzpicture}
\caption{The LOT corresponding to the presentation $\langle a,b,c,d \ | \ cac^{-1} = d, cdc^{-1} = b, ada^{-1} = c\rangle$.}
\label{Fig:LOTExample}
\end{figure}
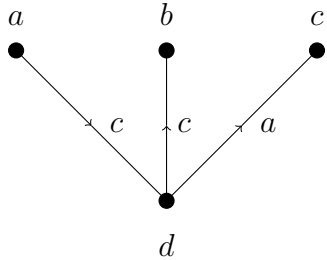

It was shown by Howie \cite{Howie:Whitehead, Howie:RibbonDiscs} that LOTs are decisively connected to the Whitehead Conjecture \cite{Whitehead:Conjecture}, which claims that every subcomplex of an aspherical 2-complex is aspherical. Recall in this context that a 2-complex is aspherical if its second homotopy group is trivial. In particular, Howie showed that if every 2-complex corresponding to a LOT  presentation (see Section \ref{Sec:Preliminaries} for details on the associated 2-complex) is aspherical, then the Andrews-Curtis Conjecture \cite{Andrews:Curtis:Conjecture} implies the finite Whitehead Conjecture (see Section \ref{Sec:Whitehead} for details).\\

Following a similar construction of Ivanov \cite{Ivanov}, the \textit{complexity} $\cp(\Gamma)$ of a LOG $\Gamma$ was introduced by Rosebrock in \cite{Rosebrock:LOTComplexity} as an invariant carrying information about the second homotopy group of the 2-complex associated to $\Gamma$. Particularly, Rosebrock showed as a main result in \cite{Rosebrock:LOTComplexity} that every LOT $\Gamma$ satisfying $\cp(\Gamma) = 2$ has an aspherical 2-complex.

Essentially, the complexity $\cp(\Gamma)$ is the minimal cardinality of a subset $S \subseteq V(\Gamma)$ of the vertices of $\Gamma$, which allows to reach every vertex of $\Gamma$ by  
\begin{enumerate}
 \item starting at the vertices in $S$ and 
 \item only passing edges, which are labeled with vertices in $S$ or formerly visited vertices.
\end{enumerate}
We give a formal definition in Section \ref{Sec:Complexity}.

It was conjectured by Rosebrock\footnote{We remark that in the final, printed version of the article \cite{Rosebrock:LOTComplexity} the conjecture is already announced as solved by the first author in his master thesis (Diplomarbeit) \cite{Christmann:Diplomarbeit}.} in \cite{Rosebrock:LOTComplexity} (and shown for LOTs with vertices of valency at most two) that the complexity $\cp(\Gamma)$ of an (interior reduced) LOT $\Gamma$ with $m$ vertices is bounded from above by $\frac{m + 1}{2}$. In the first part of the paper, we prove the conjecture. We do this in a constructive way, i.e., we give an algorithm, which yields for every $\Gamma$ a set $S \subset V(\Gamma)$ such that such that every vertex is reachable from $S$ and $|S| \leq \frac{m + 1}{2}$.

\begin{thm}[Rosebrock's Conjecture]
Let $\Gamma$ be an interior reduced, connected LOG with $m$ vertices. Then the complexity is bounded from above by
\begin{eqnarray}
 \cp(\Gamma) & \leq & \frac{m + 1}{2} \, . \label{Equ:ComplexityLOTs}
\end{eqnarray}
\label{Thm:UpperBound}
\end{thm}

In the second part of this article, we explicitly describe the LOTs for which the upper bound given in Theorem \ref{Thm:UpperBound} is satisfied with equality; see Theorem \ref{Thm:MaximalComplexity}. In particular, we show that the 2-complexes corresponding to such LOTs are always aspherical. Thus, LOTs of maximal complexity cannot disprove the Whitehead Conjecture.

\begin{thm}
The bound given in Theorem \ref{Thm:UpperBound} is sharp. In particular, if a LOT $\Gamma$ satisfies \eqref{Equ:ComplexityLOTs} with equality, then its corresponding 2-complex is aspherical.
\label{Thm:Aspherical}
\end{thm}

The article is organized as follows. In Section \ref{Sec:Preliminaries}, we fix some notation. In Section \ref{Sec:Whitehead}, we recall the connection of LOTs to the Whitehead Conjecture. In Section \ref{Sec:Complexity}, we prove the Theorems \ref{Thm:UpperBound} and Theorem \ref{Thm:Aspherical}. Excluding some cited technical statements, all proofs are purely combinatorial.\\

The content of this article was part of the master thesis (Diplomarbeit) \cite{Christmann:Diplomarbeit} of the first author.

\section*{Acknowledgements}
We thank Wolfgang Metzler for his support on the development of this article. We thank Martina Juhnke-Kubitzke, Chris O'Neill and Stephan Rosebrock for their many helpful comments.

The second author was partially supported by DFG grant TH 1333/2-1 and DFG grant MA 4797/3-2.

\section{Preliminaries}
\label{Sec:Preliminaries}

In this section we fix some notation. We begin with graphs. For additional background on graph theory see \cite{Diestel}. For a given graph $\Gamma$ we denote its \textit{vertex set} as $V(\Gamma)$ and its \textit{edge set} as $E(\Gamma)$. The cardinalities of these sets are denoted as $|V(\Gamma)|$ and $|E(\Gamma)|$, respectively. Although the graphs we investigate come with an orientation, it will be of no need for our purposes. Thus, we will restrict ourselves to the undirected version of a given graph in Section \ref{Sec:Complexity}. If $e \in E(\Gamma)$, then $v(e) = \{x,y\} \subset V(\Gamma)$ denotes the set of vertices incident to $e$ and we write $e = (x,y)$. Furthermore, we assume our graph is \textit{simple}, i.e., we always assume that no vertex is adjacent to itself and that two vertices are adjacent via at most one edge. If two vertices $x$ and $y$ are adjacent, we also denote the corresponding edge as $(x,y)$.

For our purposes, a \textit{labeling} of a graph $\Gamma$ is a function $l: E(\Gamma) \to V(\Gamma)$. For a path $\gamma = (V(\gamma),E(\gamma)) \subset (V(\Gamma),E(\Gamma))$ given by a tuple of vertices $V(\gamma)$ and a tuple of edges $E(\Gamma)$ we want to talk about labelings as tuples, not as sets. Thus, for us, paths are always directed. With slight abuse of notation we denote $v(\gamma) \in V(\Gamma) \times V(\Gamma)$ as the pair of start- end endpoint of $\gamma$ and $l(\gamma) \in V(\Gamma)^{|E(\gamma)|}$ as the tuple of length $\# \gamma := |E(\gamma)|$ containing the label of the $j$-th edge of $\gamma$ as $j$-th entry.\\

Let $G$ be a finitely presented group with presentation $P := \langle x_1,\ldots,x_n \ | \ R_1,\ldots,R_s \rangle$. For additional background on combinatorial group theory, see \cite{Stillwell}. If every relation $R_j$ satisfies $R_j = x_k x_i x_k^{-1} = x_j$, then $P$ induces a \textit{labeled oriented graph (LOG)} $\Gamma(P)$ as shown in the introduction. If the context is clear, we simply write $\Gamma$.
Analogously, every LOT $\Gamma$ induces a LOT presentation $P(\Gamma)$ in the upper sense.

Recall that every presentation $P$ induces a CW-complex, more specific a standard 2-complex, $C(P)$, where, next to one $0$ cell, every generator $x_i$ corresponds to a $1$-cell and every relation $R_j$ given by a word $x_{j_1} \cdots x_{j_k} = 1$ to a $2$-cell $D_j^2$ satisfying $\partial D_j^2 = x_{j_1} \cdots x_{j_k}$. Thus, every LOT $\Gamma$ has an associated standard 2-complex $C(P(\Gamma))$. 

Recall that a CW-complex $C$ is called \textit{aspherical} if its second homotopy group $\pi_2(C)$ is trivial, i.e., every continuous map $S^2 \to C$ is homotopy equivalent to a map $S^2 \to {p}$  with $p \in C$. If a 2-complex $C(P(\Gamma))$ associated to a LOT $\Gamma$ is aspherical, then, for convenience, we also say the LOT $\Gamma$ is aspherical. For more information about CW-complexes and 2-complexes respectively, see \cite{Hatcher,Stillwell}.

\section{Labeled Oriented Trees and the Whitehead Conjecture}
\label{Sec:Whitehead}

In this section we briefly recall the connection between LOTs and the Whitehead conjecture. It does not contain new results, but motivates them (particularly Corollary \ref{Cor:MaxComplLOTAspherical}) and sets as a brief survey of the literature. For a more detailed overview, see the survey \cite{Rosebrock:WhiteheadOverview}.\\

\noindent We start with the Whitehead Conjecture itself.

\begin{conj}(Whitehead \cite{Whitehead:Conjecture}, 1941)
Let $L$ be an aspherical 2-complex and let $K$ be a subcomplex of $L$. Then $K$ is aspherical.
\label{Conj:Whitehead}
\end{conj}

In 1983 Howie showed that if the conjecture was wrong, then there are two specific ways, how it can be falsified.

\begin{thm}(Howie \cite[Theorem 3.4.]{Howie:Whitehead})
If Conjecture \ref{Conj:Whitehead} is false, then there exists a counterexample $K \subset L$ satisfying one of the following two conditions:
\begin{enumerate}
 \item $L$ is finite and contractible, and $K = L \setminus e$ for some 2-cell $e$ of $L$ or
 \item $L$ is the union of an infinite ascending chain of finite, non-aspherical subcomplexes $K = K_0 \subset K_1 \subset \cdots$ such that each inclusion map $K_i \to K_{i+1}$ is null-homotopic.
\end{enumerate}
\label{Thm:HowieCounterexampeType}
\end{thm}

It was shown by Luft \cite{Luft} that if there exists a counterexample for the Whitehead Conjecture, then there exists one of Type (2).

However, Type (1) is more interesting and accessible for our purposes. Due to the latter statement we denote the case that no counterexample of Type (1) exists as the \textit{Finite Whitehead Conjecture}. We concentrate on this case in the following.

\begin{conj}(Finite Whitehead Conjecture)
Let $L$ be a finite, aspherical 2-complex and let $K$ be a subcomplex of $L$. Then $K$ is aspherical.
\label{Conj:FiniteWhitehead}
\end{conj}

A \textit{3-deformation} of a $2$-complex $K$ to a $2$-complex $K'$ is, roughly spoken, given by successively gluing finitely many $3$-balls up to an open $2$-cell to the boundary of the original $2$-complex respectively by doing the inverse operation. On the group theoretical side such $3$-deformations of a $2$-complex associated to a group presentation correspond to particular transformations of this presentation, which are called $Q^{**}$ \textit{transformations}. For convenience of the reader we omit detailed definitions, which are not needed for the remaining article. For more information see \cite{Hog-Angeloni:Metzler,Wright}.

A priori it is unclear how strong the assumption is that a finite, contractible 2-complex can be 3-deformed to a point. However, the Andrews-Curtis conjecture claims precisely that this is always possible.

\begin{conj}(Andrews, Curtis \cite{Andrews:Curtis:Conjecture})
Let $L$ be a finite, contractible 2-complex. Then $L$ 3-deforms to a single vertex.
\end{conj}

In particular, it was shown by Howie \cite[Theorem 4.2.]{Howie:Whitehead} that if the Andrews-Curtis Conjecture \cite{Andrews:Curtis:Conjecture} is true, then any 2-complex of the form $K = L \setminus e$, where $L$ is a finite contractible 2-complex and $e$ a 2-cell of $L$, has the simple homotopy type of the complement of a \textit{ribbon disc} (we omit the precise definition; see \cite{Howie:Whitehead} for further information). Thus, on the one hand ribbon discs are closely related to the Finite Whitehead Conjecture due to Theorem \ref{Thm:HowieCounterexampeType}. On the other hand, it was shown by Howie \cite[Propositions 3.1. and 3.2.]{Howie:RibbonDiscs} that ribbon discs are closely related to LOTs. 

For convenience of the reader we subsume the Theorem \ref{Thm:HowieCounterexampeType} and Howie's results \cite[Theorem 4.2.]{Howie:Whitehead}, and \cite[Propositions 3.1. and 3.2.]{Howie:RibbonDiscs} and obtain the following corollary, which motivates studying the topology of LOT complexes; see also \cite[Corollary 4.2.]{Rosebrock:WhiteheadOverview}.

\begin{cor}
If the Andrews-Curtis Conjecture is true and all LOTs are aspherical, then the finite Whitehead Conjecture \ref{Conj:FiniteWhitehead} is true (i.e., no counterexample of Type (1) in Theorem \ref{Thm:HowieCounterexampeType} exists).
\end{cor}

The set of possible counterexamples has been reduced in the past. Particularly, in \cite{Rosebrock:LOTComplexity} Rosebrock proved that all LOTs of complexity two are aspherical, see also Lemma \ref{Thm:LOTComplexity2Aspherical}. Furthermore, a is LOG \textit{injective} if its labeling function is injective, i.e., every vertex is assigned at most once to an edge as a label. Recently, Harlander and Rosebrock showed in \cite{Harlander:Rosebrock} that every injective LOT is aspherical.

A consequence of our results in the following section is that every LOT with maximal complexity is aspherical, see Theorems \ref{Thm:UpperBound} and \ref{Thm:Aspherical}.

\section{The Complexity of Labeled Oriented Trees}
\label{Sec:Complexity}

In this section we prove our main Theorems \ref{Thm:UpperBound} and \ref{Thm:Aspherical}. We begin with a formal definition of the complexity of a labeled oriented graph.

Let $\Gamma$ be a LOG and $S \subseteq V(\Gamma)$. We define the set $T_S \subseteq V(\Gamma)$ of \textit{reachable} (Rosebrock calls them, following Ivanov \cite{Ivanov}, ``good'') vertices (from $S$) recursively as follows.
\begin{enumerate}
 \item If $x_i \in S$, then $x_i$ is reachable.
 \item If $x_i$ is incident to an edge $e = (x_i,x_j)$ and $x_j$ as well as the label $l(e)$ are reachable, then $x_i$ is also reachable (the orientation does not matter).
\end{enumerate}
If every vertex is reachable from $S$, i.e., $T_S = V(\Gamma)$, then we say $\Gamma$ is \textit{reachable} (from $S$).\\

In \cite{Rosebrock:LOTComplexity} Rosebrock defines the \textit{complexity} $\cp(\Gamma) \in \N$ of $\Gamma$ as the minimum of the cardinalities of all sets $S \subseteq V(\Gamma)$ such that $\Gamma$ is reachable from $S$. See Figure \ref{Fig:Reachable} for an example. Since the orientations of the edges have no impact on the complexity, we omit it from here on and investigate the corresponding undirected graph.

\begin{figure}
\begin{tikzpicture}
   \draw (0,0) -- (2,0);
   \draw (2,0) -- (2,2);
   \draw [very thick, blue] (2,2) -- (2,4);
   \draw (0,2) -- (2,0);
   \draw (2,0) -- (4,2);
   
   \draw [fill] (0,0) circle [radius=0.1];
   \draw [fill] (2,0) circle [radius=0.1];
   \draw [fill] (0,2) circle [radius=0.1];
   \draw [fill=red] (2,4) circle [radius=0.2];
   \draw [fill=red] (4,2) circle [radius=0.2];
   \draw [fill=blue] (1.85,1.85) rectangle (2.15,2.15);
   
   \node [above, red] at (2,4.3) {$x_1$};
   \node [above] at (0,2.3) {$x_2$};
   \node [left, above, blue] at (1.7,2.2) {$x_3$};
   \node [above, red] at (4,2.3) {$x_4$};
   \node [below] at (0,-0.3) {$x_5$};
   \node [below] at (2,-0.3) {$x_6$};
   
   \node [right, red] at (2,3) {$x_4$};
   \node [right] at (2,1) {$x_2$};
   \node [right, above] at (1.1,1.1) {$x_1$};
   \node [right, below] at (3.1,0.9) {$x_5$};
   \node [below] at (1,0) {$x_4$};

   \draw (5,0) -- (7,0);
   \draw [very thick, blue] (7,0) -- (7,2);
   \draw (7,2) -- (7,4);
   \draw [very thick, blue] (5,2) -- (7,0);
   \draw (7,0) -- (9,2);
   
   \draw [fill] (5,0) circle [radius=0.1];
   \draw [fill=red] (7,0) circle [radius=0.2];
   \draw [fill=blue] (4.85,1.85) rectangle (5.15,2.15);
   \draw [fill=red] (7,4) circle [radius=0.2];
   \draw [fill] (9,2) circle [radius=0.1];
   \draw [fill=blue] (6.85,1.85) rectangle (7.15,2.15);
   
   \node [above, red] at (7,4.3) {$x_1$};
   \node [above, blue] at (5,2.3) {$x_2$};
   \node [left, above, blue] at (6.7,2.2) {$x_3$};
   \node [above] at (9,2.3) {$x_4$};
   \node [below] at (5,-0.3) {$x_5$};
   \node [below, red] at (7,-0.3) {$x_6$};
   
   \node [right] at (7,3) {$x_4$};
   \node [right, blue] at (7,1) {$x_2$};
   \node [right, above, red] at (6.1,1.1) {$x_1$};
   \node [right, below] at (8.1,0.9) {$x_5$};
   \node [below] at (6,0) {$x_4$};

   \draw [very thick, blue] (10,0) -- (12,0);
   \draw [very thick, blue] (12,0) -- (12,2);
   \draw [very thick, blue] (12,2) -- (12,4);
   \draw [very thick, blue] (10,2) -- (12,0);
   \draw [very thick, blue] (12,0) -- (14,2);
   
   \draw [fill=blue] (9.85,-0.15) rectangle (10.15,0.15);
   \draw [fill=red] (12,0) circle [radius=0.2];
   \draw [fill=blue] (9.85,1.85) rectangle (10.15,2.15);
   \draw [fill=red] (12,4) circle [radius=0.2];
   \draw [fill=red] (14,2) circle [radius=0.2];
   \draw [fill=blue] (11.85,1.85) rectangle (12.15,2.15);
   
   \node [above, red] at (12,4.3) {$x_1$};
   \node [above, blue] at (10,2.3) {$x_2$};
   \node [left, above, blue] at (11.7,2.2) {$x_3$};
   \node [above, red] at (14,2.3) {$x_4$};
   \node [below, blue] at (10,-0.3) {$x_5$};
   \node [below,red] at (12,-0.3) {$x_6$};
   
   \node [right, red] at (12,3) {$x_4$};
   \node [right, blue] at (12,1) {$x_2$};
   \node [right, above, blue] at (11.1,1.1) {$x_1$};
   \node [right, below, blue] at (13.1,0.9) {$x_5$};
   \node [below, red] at (11,0) {$x_4$};
\end{tikzpicture}
\caption{A LOT $\Gamma$ with three different sets $S \subset V(\Gamma)$ (red big circles) and their reachable sets $T_S \subseteq V(\Gamma)$ (red big circles plus blue squares). Only for the right set $S$ the entire LOT $\Gamma$ is reachable.}
\label{Fig:Reachable}
\end{figure}
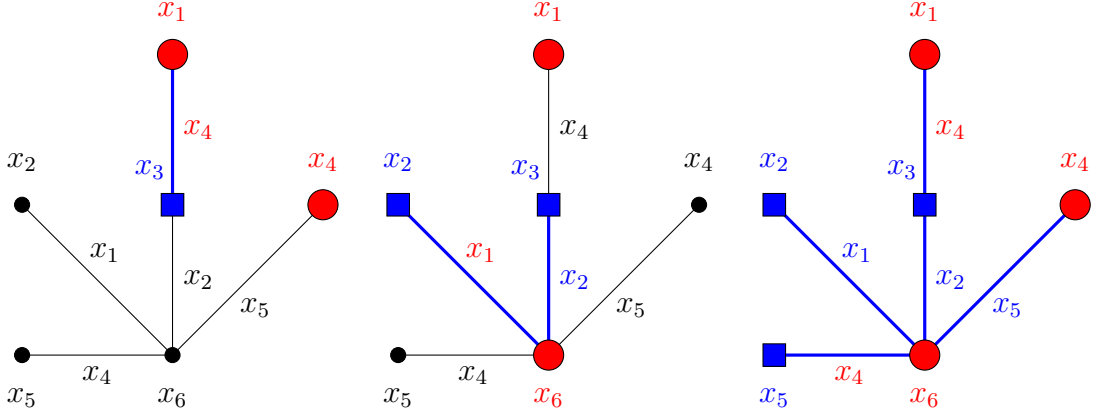

Obviously, the complexity of every interior reduced LOG $\Gamma$ is bounded from below by $2$. Furthermore, it is obviously bounded from above by the minimum of $|V(\Gamma)|$ and 1 plus the cardinality of the image of the labeling $l$ of $\Gamma$.\\

In order to prove our first main result, we will need the following technical statement.

\begin{lemma}
Let $\Gamma$ be a connected, interior reduced labeled oriented tree and $S \subset V(\Gamma)$. For every $x_i \in V(\Gamma)$ holds that if $V(\Gamma) \setminus \{x_i\}$ is reachable from $S$, then $V(\Gamma)$ is reachable from $S$.
\label{Lem:AllReachable}
\end{lemma}

\begin{proof}
Let $V(\Gamma) \setminus \{x_i\}$ be reachable from $S$ for some $x_i \in V(\Gamma)$. Since $\Gamma$ is connected, there exists an edge $e = (x_i,x_j)$ and label $l(e) = x_k$ for some $x_j \in V \setminus \{x_i\}$ and $x_k \in V$. Since $\Gamma$ is interior reduced, we know $x_k \neq x_i$. Thus, $x_j$ and $x_k$ are reachable, since $V(\Gamma) \setminus \{x_i\}$ is reachable. And since $e$ represents the relation $x_k x_j x_k^{-1} = x_i$ (up to orientation), $x_i$ is reachable.
\end{proof}

Now, we can prove Theorem \ref{Thm:UpperBound}, stating that the complexity of an interior reduced, connected LOG with $m$ vertices is bounded from above by $\frac{m + 1}{2}$. It suffices to prove the statement for LOTs since every connected graph has a minimal spanning tree and additional edges can only decrease the complexity. The proof is constructive. We give an algorithm, that successively constructs a subset $S$ of the generators of $P$ such that in the end $S$ has cardinality at most $\lfloor (m+1)/2 \rfloor$ and such that every vertex $x_j$ in $\Gamma$ can be visited from a vertex in $x_i \in S$ along a path only using edges labeled with elements in $S$ or vertices, which were visited before in the successive steps.

\begin{proof}(Theorem \ref{Thm:UpperBound})
We construct $S$ successively as $S_1 \subset S_2 \subset \cdots \subset S$, where we set $S_1 := \{x_1\}$ for some arbitrary generator $x_1$. As an outline, the key idea of the proof is to choose in every step a generator, which, say, gives at least one additional generator for free. I.e., we choose a generator, which allows us to reach at least one additional generator, which we had not chosen before.

We denote $T_k \subset V(\Gamma)$ as the set of generators $x_i \in V(\Gamma)$ reachable from $S_k$. If after $k$ steps $T_k = V(\Gamma)$, then we set $S := S_k$ and stop. Otherwise, we define $S_{k+1} := S_k \cup \{x_i\}$ for some $x_i$ such that
\begin{enumerate}
 \item $x_i \notin T_k$ and
 \item there exists an edge $e \in E(\Gamma)$ such that
 \begin{enumerate}
  \item  $l(e) = x_i$ and
  \item for $v(e) = \{x_j,x_j'\}$ holds: $x_j \in T_k$ and $x_j' \in V(\Gamma) \setminus T_k$.
 \end{enumerate}
\end{enumerate}

Such an $x_i$ always exists: since $T_k \neq V(\Gamma)$ and $\Gamma$ is connected we find an edge $e \in E(\Gamma)$ connecting some $x_j \in T_k$ with some $x_j' \in V(\Gamma) \setminus T_k$. But, if we had now $l(e) \in T_k$, then $x_j'$ was reachable from $S_k$ by definition of $T_k$, and thus we had $x_j' \in T_k$, which is a contradiction. Hence, $l(e) \in V(\Gamma) \setminus T_k$ and we can set $x_i := l(e)$. Therefore, with $S_{k+1} := S_k \cup \{x_i\}$ we obtain $T_{k+1} \subseteq T_k \cup \{x_i,x_j'\}$ where $T_k \cap \{x_i,x_j'\} = \emptyset$. Furthermore, $x_i \neq x_j'$ since $\Gamma$ is interior reduced.

This construction ensures that for every $k > 1$, the cardinality of the set $T_k$ increases by at least two but the cardinality of $S_k$ increases only by one. Hence, for $m$ odd $T_{(m+1)/2}$ equals $V(\Gamma)$, and for $m$ even $T_{m/2}$ equals $V(\Gamma)$ with Lemma \ref{Lem:AllReachable}. I.e., all elements are reachable from $S_{(m+1)/2}$ respectively $S_{m/2}$ and, since the cardinality of $S_k$ equals $k$, we have $\cp(\Gamma) \leq (m + 1)/2$.
\end{proof}

In order to tackle the second main statement, Theorem \ref{Thm:Aspherical}, which states that the upper bound for the complexity is sharp and LOTs of maximal complexity are aspherical, we need to introduce some more notation.

\begin{defn}
Let $\Gamma$ and $\Gamma_1,\Gamma_2$ be LOTs. We say that $\Gamma$ is \textit{decomposable} into $\Gamma_1,\Gamma_2$, noted as
\begin{eqnarray}
 \Gamma & = & \Gamma_1 \sqcup \Gamma_2, \label{Equ:Decomposition}
\end{eqnarray}
if $\Gamma_1$ and $\Gamma_2$ have distinct vertices and labels and $\Gamma$ is obtained by identifying one vertex of $\Gamma_1$ with one vertex of $\Gamma_2$. If such a decomposition exists, then we say $\Gamma_1$ and $\Gamma_2$ are contained in $\Gamma$.
\end{defn}

Note that group theoretically that means if $\Gamma_1$ corresponds to a presentation  $P_1 = \langle x_1,\ldots,x_n \ | \ R_1,\ldots,R_k \rangle$ and $\Gamma_2$ to a presentation $P_2 = \langle y_1,\ldots,y_m \ | \ S_1,\ldots,S_l \rangle$, then $\Gamma = \Gamma_1 \sqcup \Gamma_2$ has the presentation
\begin{eqnarray*}
 P & = & \langle x_1,\ldots,x_n, y_1,\ldots,y_m \ | \ R_1,\ldots,R_k, S_1,\ldots,S_l, x_i = y_j \rangle,
\end{eqnarray*}
where $x_i$ and $y_j$ are the generators, which are identified. This presentation can obviously be transformed into at LOT presentation
\begin{eqnarray*}
 P' & = & \langle x_1,\ldots,x_n, y_1,\ldots,y_{j-1}, y_{j+1},\ldots,y_m \ | \ R_1,\ldots,R_k, S_1',\ldots,S_l' \rangle,
\end{eqnarray*}
where each $S_r'$ is obtained from $S_r$ by replacing $y_j$ by $x_i$. We remark that the transformation used here is actually a $Q^{**}$ transformation, which corresponds to a 3-deformation in the associated 2-complex; see for example \cite{Christmann:Diplomarbeit, Hog-Angeloni:Metzler} for further details. Thus, it is in particular not violating the assumptions in Howie's Theorems, see Section \ref{Sec:Whitehead}.

It is easy to see that the group $G$ presented by $P$ respectively $P'$ is an amalgam of the groups $G_1$ and $G_2$ presented by $P_1$ and $P_2$ with the free group generated by $x_i$ as joint subgroup, see \cite{Stillwell}.\\

Similarly, we say that a LOT $\Gamma$ is \textit{decomposable} into LOTs $\Gamma_1,\ldots,\Gamma_s$, noted as
\begin{eqnarray}
 \Gamma & = & \Gamma_1 \sqcup \cdots \sqcup \Gamma_s, \label{Equ:Decomposition2}
\end{eqnarray}
if $\Gamma$ can be successively decomposed in the sense of \eqref{Equ:Decomposition}, i.e.,
\begin{eqnarray}
 \Gamma & = & (\cdots((\Gamma_1 \sqcup \Gamma_2) \sqcup \Gamma_3) \cdots \Gamma_{s-1}) \sqcup \Gamma_s. \label{Equ:Decomposition3}
\end{eqnarray}
Note that in this case $G$ is not necessarily an amalgam of $G_1,\ldots,G_s$ anymore, since the process of identifying generators is not transitive. Instead, $G$ is the result of a recursive process of taking amalgams. But it is an immediate consequence of the LOT representation $\Gamma$ and $\Gamma_1,\ldots,\Gamma_s$ that the process of identifying vertices respectively generators is associative. So we can safely leave out parenthesis in \eqref{Equ:Decomposition3} and the Notation \eqref{Equ:Decomposition2} is well defined. See Figure \ref{Fig:Decomposable} for a depiction of decomposition of LOTs.\\

\begin{figure}
 \begin{tikzpicture}[scale=1]
   \draw [blue] (0,2) -- (2,2);
   \draw [blue] (0,0) -- (0,2);
   \draw [blue] (0,2) -- (0,4);
   
   \draw [fill, blue] (2,2) circle [radius=0.1];
   \draw [fill, blue] (0,2) circle [radius=0.1];
   \draw [fill, blue] (0,0) circle [radius=0.1];
   \draw [fill, blue] (0,4) circle [radius=0.1];
   
   \node [left, blue] at (0,0) {$x_1$};
   \node [left, blue] at (0,2) {$x_2$};
   \node [left, blue] at (0,4) {$x_3$};
   \node [above, blue] at (2,2) {$x_4$};

   \draw [red] (1.5,0) -- (3,1.5);
   \draw [red] (4.5,0) -- (3,1.5);
   \draw [red] (3,1.5) -- (3,5.5);
   
   \draw [fill, red] (1.5,0) circle [radius=0.1];
   \draw [fill, red] (3,1.5) circle [radius=0.1];
   \draw [fill, red] (4.5,0) circle [radius=0.1];
   \draw [fill, red] (3,3.5) circle [radius=0.1];
   \draw [fill, red] (3,5.5) circle [radius=0.1];
   
   \node [below, red] at (1.5,-0.2) {$x_5$};
   \node [right, red] at (3.2,1.5) {$x_6$};
   \node [right, red] at (3.2,3.5) {$x_7$};
   \node [right, red] at (3.2,5.5) {$x_8$};
   \node [below, red] at (4.5,-0.2) {$x_9$};

   \draw [DarkGreen] (6,0) -- (6,4);
   
   \draw [fill, DarkGreen] (6,0) circle [radius=0.1];
   \draw [fill, DarkGreen] (6,2) circle [radius=0.1];
   \draw [fill, DarkGreen] (6,4) circle [radius=0.1];
   
   \node [right, DarkGreen] at (6.2,0) {$x_{10}$};
   \node [right, DarkGreen] at (6.2,2) {$x_{11}$};
   \node [right, DarkGreen] at (6.2,4) {$x_{12}$};

   \draw [<->, line width = 3pt] (7.5,2) -- (9,2);

   \draw [blue] (11,3) -- (12.5,1.5);
   \draw [blue] (11,1) -- (11,3);
   \draw [blue] (11,3) -- (11,5);
   
   \draw [red] (11,0) -- (12.5,1.5);
   \draw [red] (14,0) -- (12.5,1.5);
   \draw [red] (12.5,1.5) -- (12.5,5.5);
   
   \draw [DarkGreen] (12.5,3.5) -- (14.5,3.5);
   \draw [DarkGreen] (14.5,3.5) -- (14.5,1.5);
   
   \draw [fill, blue] (11,1) circle [radius=0.1];
   \draw [fill, blue] (11,3) circle [radius=0.1];
   \draw [fill, blue] (11,5) circle [radius=0.1];
   
   \node [left, blue] at (11,1) {$x_1$};
   \node [left, blue] at (11,3) {$x_2$};
   \node [left, blue] at (11,5) {$x_3$};
   
   \draw [fill, red] (11,0) circle [radius=0.1];
   \draw [fill, MyPurple] (12.5,1.5) circle [radius=0.17];
   \draw [fill, red] (14,0) circle [radius=0.1];
   \draw [fill, MyYellow] (12.5,3.5) circle [radius=0.17];
   \draw [fill, red] (12.5,5.5) circle [radius=0.1];
   
   \node [below, red] at (11,-0.2) {$x_5$};
   \node [right, MyPurple] at (12.7,1.5) {$x_4 = x_6$};
   \node [right, MyYellow] at (12.7,3.8) {$x_7 = x_{12}$};
   \node [right, red] at (12.7,5.5) {$x_8$};
   \node [below, red] at (14,-0.2) {$x_9$};
   
   \draw [fill, DarkGreen] (14.5,3.5) circle [radius=0.1];
   \draw [fill, DarkGreen] (14.5,1.5) circle [radius=0.1];
   
   \node [right, DarkGreen] at (14.7,3.5) {$x_{11}$};
   \node [right, DarkGreen] at (14.7,1.5) {$x_{10}$};
\end{tikzpicture}
\caption{A LOT, which is decomposable in three single LOTs via two identified vertices. The edge labels are left out for clarity of the figure.}
\label{Fig:Decomposable}
\end{figure}
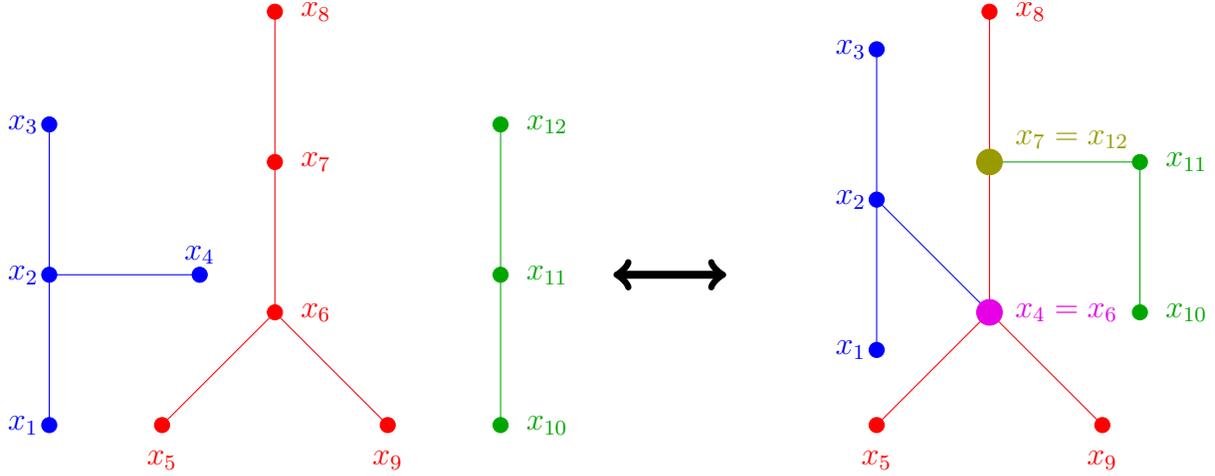

Motivated by constructions of Rosebrock used in \cite{Rosebrock:LOTComplexity} we make the following definition.

\begin{defn}
We call a LOT \textit{Rosebrock} if it is has three vertices and is of the form as shown in Figure \ref{Fig:Rosebrock3LOI}.
\label{Defn:RosebrockLOT}
\end{defn}

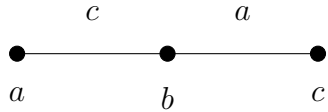
\begin{figure}
  \begin{tikzpicture}
   \draw (0,0) --(4,0);
   \draw [fill] (0,0) circle [radius=0.1];
   \draw [fill] (2,0) circle [radius=0.1];
   \draw [fill] (4,0) circle [radius=0.1];
   \node [below] at (0,-0.3) {$a$};
   \node [below] at (2,-0.3) {$b$};
   \node [below] at (4,-0.3) {$c$};
   \node [above] at (1,0.3) {$c$};
   \node [above] at (3,0.3) {$a$};
  \end{tikzpicture}
\caption{The Rosebrock LOT.}
\label{Fig:Rosebrock3LOI}
 \end{figure}

As a main step towards Theorem \ref{Thm:Aspherical} we show the following statement, which characterizes the LOTs which have maximal complexity.
 
\begin{thm}
Let $\Gamma$ be a connected, interior reduced LOT with $m$ vertices. Then
\begin{eqnarray*}
 \cp(\Gamma) &  = & \frac{m + 1}{2}
\end{eqnarray*}
if and only if $\Gamma$ is decomposable as
\begin{eqnarray}
 \Gamma & = & \Gamma_1 \sqcup \cdots \sqcup \Gamma_s, \label{Equ:RosebrockLotDecomp}
\end{eqnarray}
such that every $\Gamma_i$ is a Rosebrock LOT.
\label{Thm:MaximalComplexity}
\end{thm}

Observe that a LOT with a decomposition into $s$ Rosebrock LOTs as in \eqref{Equ:RosebrockLotDecomp} has exactly $2s + 1$ vertices and $2s$ edges. Furthermore, the following lemma holds, which is needed for the proof of the Theorem \ref{Thm:MaximalComplexity}.

\begin{lemma}
Let $\Gamma$ be a connected, interior reduced LOT with $2s + 1$ vertices. If every edge $e$ of $\Gamma$ is an edge of a Rosebrock LOT $\Gamma_i$ contained in $\Gamma$, then there exists a unique decomposition
\begin{eqnarray}
 \Gamma & = & \Gamma_1 \sqcup \cdots \sqcup \Gamma_s,
\end{eqnarray}
such that every $\Gamma_i$ is a Rosebrock LOT.
\label{Lem:UniqueDecomposition}
\end{lemma}

This means that the existence of a decomposition of the Form \eqref{Equ:RosebrockLotDecomp} into Rosebrock LOTs is a local property of the graph.

\begin{proof}(Lemma \ref{Lem:UniqueDecomposition})
We prove the statement by strong induction over $s$. If $s = 1$, then $\Gamma$ has two edges and since $\Gamma$ is interior reduced, it has to be Rosebrock; see Figure \ref{Fig:Rosebrock3LOI}. Assume the statement holds for all interior reduced LOTs with at most $2s$ edges and let $|E(\Gamma)| = 2(s + 1)$. We investigate an edge $e_1$ connecting a leaf $a$ of $\Gamma$ with some vertex $b$. Since $\Gamma$ is interior reduced, we have $l(e_1) = c$ with $c \notin \{a,b\}$. Since $e_1$ is an edge of a Rosebrock LOT $\Gamma_1$ contained in $\Gamma$, it follows from Definition \ref{Defn:RosebrockLOT} that there exists an edge $e_2$ connecting the vertices $b$ and $c$ and satisfying $l(e_2) = a$; see Figure \ref{Fig:Rosebrock3LOI} once more. Thus, we can make a decomposition
\begin{eqnarray*}
 \Gamma & = & \Gamma_1 \sqcup \Gamma_b \sqcup \Gamma_c,
\end{eqnarray*}
where $\Gamma_b$ and $\Gamma_c$ are the subtrees of $\Gamma$ containing $b$ respectively $c$ (one tree is possibly just one vertex). We investigate an arbitrary edge $e_3$ in $\Gamma_b$. Since $e_3$ is edge of $\Gamma$ it is the edge of a Rosebrock LOT $\Gamma_2$.

\textbf{Claim}: $\Gamma_2$ is contained in $\Gamma_b$. Namely, $\Gamma_b$ and $\Gamma_c$ are only connected via the edge $e_2$, since $\Gamma$ is a tree. Furthermore, $\Gamma_2$ is a Rosebrock LOT in $\Gamma$. Thus, if $\Gamma_2$ was not contained in $\Gamma_b$, then the only possibility is that the second edge of $\Gamma_2$ is $e_1$ or $e_2$. But then $e_3$ would connect the vertices $b$ and $a$ or $b$ and $c$, which is impossible, since $e_3$ is not contained in $\Gamma_1$ and every vertex appears only once. Thus, $\Gamma_2$ is contained in $\Gamma_b$.

But since $\Gamma_b$ furthermore has at most $2s$ edges and is an interior reduced LOT (since $\Gamma$ is an interior reduced LOT), it follows by the induction hypothesis that there exists a unique decomposition $\Gamma_b = \Gamma_{b_1} \sqcup \cdots \sqcup \Gamma_{b_k}$ such that every $\Gamma_{b_j}$ is a Rosebrock LOT. Analogously for $\Gamma_c$. Hence, we obtain a unique decomposition
\begin{eqnarray*}
\Gamma & = & \Gamma_1 \sqcup \Gamma_{b_1} \sqcup \cdots \sqcup \Gamma_{b_k} \sqcup \Gamma_{c_1} \sqcup \cdots \sqcup \Gamma_{c_l}
\end{eqnarray*}
of $\Gamma$ into Rosebrock LOTs. And since every Rosebrock LOT has two edges we have $k + l + 1 = 2s$.
\end{proof}

Now we can prove Theorem \ref{Thm:MaximalComplexity}.

\begin{proof}(Theorem \ref{Thm:MaximalComplexity})
First, we prove that if a LOT is decomposable as stated, then it has the claimed complexity. We make an induction over $s$ with respect to the decomposition described in \eqref{Equ:RosebrockLotDecomp}. Obviously, for $s = 1$, i.e., $\Gamma$ is a LOT as depicted in Figure \ref{Fig:Rosebrock3LOI}, $\Gamma$ has complexity two. Now, assuming that the claim holds true for $\Gamma = \Gamma_1 \sqcup \cdots \sqcup \Gamma_s$, we prove it holds for $\Gamma' = \Gamma \sqcup \Gamma_{s+1} = \Gamma_1 \sqcup \cdots \sqcup \Gamma_{s+1}$. Assume again without loss of generality that $\Gamma_{s+1}$ is given as in Figure \ref{Fig:Rosebrock3LOI} and that $a$ is the vertex, where $\Gamma$ and $\Gamma_{s+1}$ are identified. Then $\Gamma'$ has two more vertices than $\Gamma$, namely $2(s+1) + 1$. But for the complexity of $\Gamma'$ we have $\cp(\Gamma') = \cp(\Gamma) + 1$: On the one hand, since $a$ is reachable in $\Gamma$, both $b$ and $c$ are reachable, if we choose one of them. Hence, $\cp(\Gamma') \leq \cp(\Gamma) + 1$. On the other hand, adding either $b$ or $c$ to the set of generators does not lower the complexity of the $\Gamma$ part of $\Gamma'$, since, by assumption, no edge in $\Gamma$ has label $b$ or $c$. Thus, $\cp(\Gamma') \geq \cp(\Gamma) + 1$. But then we are done, since we obtain with the induction hypothesis
\begin{eqnarray*}
 \cp(\Gamma') \ = \ \cp(\Gamma) + 1 \ = \ \frac{(2s + 1) + 1}{2} + 1 \ = \ \frac{(2(s+1) + 1) + 1}{2} \ = \ \frac{m + 1}{2},
\end{eqnarray*}
where the last equation holds since $\Gamma_{s+1}$ has $2(s+1) + 1$ vertices and $m$ denotes the cardinality of the vertex set.\\

Now, assume that $\Gamma$ is a LOT, which is not decomposable in Rosebrock LOTs. By Lemma \ref{Lem:UniqueDecomposition} this means in particular that there exists an edge $e \in E(\Gamma)$, which is not part of a Rosebrock LOT. Let $l(e) = z$ and $v(e) = \{x,y\}$. By investigating different cases, we conclude that $\Gamma$ cannot have maximal complexity.\\

\noindent \textbf{Case 1}: $z$ is connected to $x$ or $y$ (without loss of generality: $x$) via an edge:

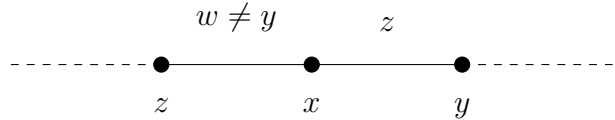
\begin{center}
 \begin{tikzpicture}
   \draw [dashed] (0,0) -- (2,0);
   \draw (2,0) -- (6,0);
   \draw [dashed] (6,0) -- (8,0);
   \draw [fill] (2,0) circle [radius=0.1];
   \draw [fill] (4,0) circle [radius=0.1];
   \draw [fill] (6,0) circle [radius=0.1];
   \node [below] at (2,-0.3) {$z$};
   \node [below] at (4,-0.3) {$x$};
   \node [below] at (6,-0.3) {$y$};
   \node [above] at (3,0.3) {$w \neq y$};
   \node [above] at (5,0.3) {$z$};
  \end{tikzpicture}
\end{center}

As depicted, the edge $e'$ connecting $z$ and $x$ with $l(e') = w$ has to satisfy $w \neq y$. Otherwise, the depicted part of $\Gamma$ would be a Rosebrock LOT (see Figure \ref{Fig:Rosebrock3LOI}), which we excluded. We start the algorithm presented in the proof of Theorem \ref{Thm:UpperBound} at $z$ and choose $w$ in the second step. With this choice we reach at least the vertices $x,y,w,z$. Afterwards, we proceed as in the proof of Theorem \ref{Thm:UpperBound}. Since we do not need to choose an initial vertex, we need to choose at most $\lfloor ((m-4) + 1)/2 \rfloor - 1$ additional vertices in order to reach every vertex of $\Gamma$. Hence, in total
\begin{eqnarray}
 \cp(\Gamma) \ \leq \ \left\lfloor \frac{m+1}{2} \right\rfloor - 1 \ < \ \frac{m+1}{2} \, . \label{Equ:ProofMaximalComplexity1}
\end{eqnarray}

\noindent \textbf{Case 2}: $z$ is not connected to $x$ or $y$ via an edge:
 
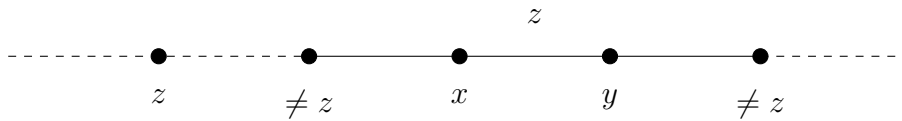
\begin{center}
 \begin{tikzpicture}
   \draw [dashed] (0,0) -- (4,0);
   \draw (4,0) -- (10,0);
   \draw [dashed] (10,0) -- (12,0);
   \draw [fill] (2,0) circle [radius=0.1];
   \draw [fill] (4,0) circle [radius=0.1];
   \draw [fill] (6,0) circle [radius=0.1];
   \draw [fill] (8,0) circle [radius=0.1];
   \draw [fill] (10,0) circle [radius=0.1];
   \node [below] at (2,-0.3) {$z$};
   \node [below] at (4,-0.3) {$\neq z$};
   \node [below] at (6,-0.3) {$x$};
   \node [below] at (8,-0.3) {$y$};
   \node [below] at (10,-0.3) {$\neq z$};
   \node [above] at (7,0.3) {$z$};
  \end{tikzpicture}
\end{center}

Since $\Gamma$ is a tree, $z$ is connected to $x$ or $y$ (without loss of generality: $x$) via a path $\gamma$ given by edges $e_1,\ldots,e_k$ such that the edge $e$ is not part of $\gamma$.

\textbf{Case 2.1}: First assume that no edge $e' \in E(\gamma) \subset E(\Gamma)$ satisfies $l(e') \in \{x,y\}$, i.e.,

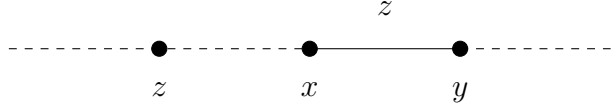
\begin{center}
 \begin{tikzpicture}
   \draw [dashed] (0,0) -- (4,0);
   \draw (4,0) -- (6,0);
   \draw [dashed] (6,0) -- (8,0);
   \draw [fill] (2,0) circle [radius=0.1];
   \draw [fill] (4,0) circle [radius=0.1];
   \draw [fill] (6,0) circle [radius=0.1];
   \node [below] at (2,-0.3) {$z$};
   \node [below] at (4,-0.3) {$x$};
   \node [below] at (6,-0.3) {$y$};
   \node [above] at (5,0.3) {$z$};
  \end{tikzpicture}
\end{center}

We choose $z$ as the initial vertex in the algorithm in the proof of Theorem \ref{Thm:UpperBound}. In steps $2$ through (at most) $k$ we choose successive vertices $l(e_1),\ldots,l(e_k)$. Without loss of generality this yields two new reachable vertices in every step (not more since otherwise the maximal complexity cannot be attained anymore). But in step $k$, by our choice of $l(e_k)$, we obtain $x$ as a reachable vertex and since $z$ is reachable, we also obtain $y$ as reachable. Since $l(e_1),\ldots,l(e_k) \notin \{x,y\}$, after $k+1$ steps we have chosen at most $k+1$ vertices but reach at least $2(k+1)$ vertices. If we proceed as in Case 1, then we obtain \eqref{Equ:ProofMaximalComplexity1}.

\textbf{Case 2.2}: Now assume that there exists some edge $e_j$ in the path $\gamma$ with $j \neq k$, $l(e_j) \in \{x,y\}$, and $l(e_k) = w \notin \{x,y\}$.

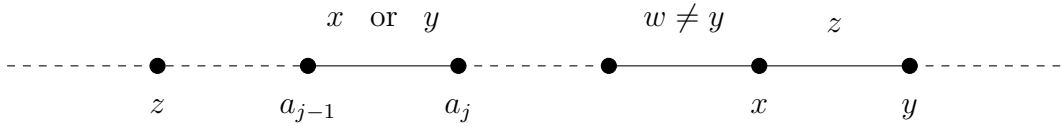
\begin{center}
 \begin{tikzpicture}
   \draw [dashed] (0,0) -- (4,0);
   \draw (4,0) -- (6,0);
   \draw [dashed] (6,0) -- (8,0);
   \draw (8,0) -- (12,0);
   \draw [dashed] (12,0) -- (14,0);
   \draw [fill] (2,0) circle [radius=0.1];
   \draw [fill] (4,0) circle [radius=0.1];
   \draw [fill] (6,0) circle [radius=0.1];
   \draw [fill] (8,0) circle [radius=0.1];
   \draw [fill] (10,0) circle [radius=0.1];
   \draw [fill] (12,0) circle [radius=0.1];
   \node [below] at (2,-0.3) {$z$};
   \node [below] at (4,-0.3) {$a_{j-1}$};
   \node [below] at (6,-0.3) {$a_j$};
   \node [below] at (10,-0.3) {$x$};
   \node [below] at (12,-0.3) {$y$};
   \node [above] at (5,0.3) {$x$ \, or \, $y$};
   \node [above] at (9,0.3) {$w \neq y$};
   \node [above] at (11,0.3) {$z$};
  \end{tikzpicture}
\end{center}

Again, we choose $z$ as the initial vertex in our algorithm and $l(e_1),\ldots,l(e_{j-1})$ successively afterwards (skipping those already reachable), which yields two new reachable vertices in every step but the first one. Since $l(e_j) \in \{x,y\}$ and $z$ was already chosen, by choosing $l(e_j)$ in step $j+1$, we get three new reachable vertices, namely, $x,y$ and $a_j$. Hence, after $j+1$ steps we have chosen at most $j+1$ vertices but reach at least $2(j+1)$ vertices. If we proceed on as in Case 1, then we obtain \eqref{Equ:ProofMaximalComplexity1}.

\textbf{Case 2.3}: Finally, assume that the final edge $e_k$ satisfies $v(e_k) = \{a_k,x\}$ and $l(e_k) = y$.

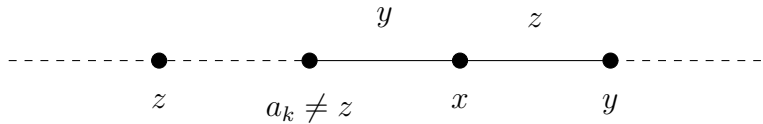
\begin{center}
 \begin{tikzpicture}
   \draw [dashed] (0,0) -- (4,0);
   \draw (4,0) -- (8,0);
   \draw [dashed] (8,0) -- (10,0);
   \draw [fill] (2,0) circle [radius=0.1];
   \draw [fill] (4,0) circle [radius=0.1];
   \draw [fill] (6,0) circle [radius=0.1];
   \draw [fill] (8,0) circle [radius=0.1];
   \node [below] at (2,-0.3) {$z$};
   \node [below] at (4,-0.3) {$a_k \neq z$};
   \node [below] at (6,-0.3) {$x$};
   \node [below] at (8,-0.3) {$y$};
   \node [above] at (5,0.3) {$y$};
   \node [above] at (7,0.3) {$z$};
  \end{tikzpicture}
\end{center}

In the first two steps of our algorithm, we choose $z$ and one of the vertices $x$ or $y$. After these two steps at least the vertices $x,y,z$ and $a_k$ are reachable. Afterwards, we proceed as in Case 1 and, again, we obtain \eqref{Equ:ProofMaximalComplexity1}.
\end{proof}

To prove the asphericity of (2-complexes of) LOTs with maximal complexity we use the following lemma by Rosebrock (see \cite{Rosebrock:LOTComplexity}; we adjusted it to our notation).

\begin{lemma}[Rosebrock]
Let $\Gamma$ be a LOT with a decomposition $\Gamma = \Gamma_1 \sqcup \Gamma_2$. If both 2-complexes corresponding to $\Gamma_1$ and $\Gamma_2$ are aspherical, then also the 2-complex corresponding to $\Gamma$ is aspherical.
\label{Lem:Amalgam}
\end{lemma}

As already mentioned the main result of Rosebrock in \cite{Rosebrock:LOTComplexity} is the asphericity of LOTs of complexity two. We recall the statement here.

\begin{thm}[Rosebrock]
Let $\Gamma$ be a LOT of complexity two. Then its corresponding 2-complex is aspherical.
\label{Thm:LOTComplexity2Aspherical}
\end{thm}

Now, we have everything needed to finish the proof of our second main Theorem \ref{Thm:Aspherical}.

\begin{cor}
Let $\Gamma$ be a LOT with $m$ vertices and maximal complexity, i.e., $\cp(\Gamma) = (m+1)/2$. Then its corresponding 2-complex is aspherical.
\label{Cor:MaxComplLOTAspherical}
\end{cor}

\begin{proof}(Corollary \ref{Cor:MaxComplLOTAspherical} / Theorem \ref{Thm:Aspherical})
Let $\Gamma$ be a LOT of maximal complexity. Thus, by Theorem \ref{Thm:MaximalComplexity}, $\Gamma$ has a decomposition $\Gamma = \Gamma_1 \sqcup \cdots \sqcup \Gamma_{s}$ such that every $\Gamma_j$ is a Rosebrock LOT. Every Rosebrock LOT has complexity 2, i.e., its corresponding 2-complex is aspherical by Theorem \ref{Thm:LOTComplexity2Aspherical}. Therefore, by Lemma \ref{Lem:Amalgam}, the 2-complex corresponding to $\Gamma$ is aspherical.
\end{proof}

\bibliographystyle{amsplain}
\bibliography{LOT_Complexity}

\end{document}